\documentclass[final]{siamltex}
\usepackage{amsfonts}
\usepackage{amssymb}
\usepackage{amsmath}
\usepackage{epsfig}
\def\a{a}
\def\b{\vec{\mathbf{\beta}}}

\title{Optimal Preconditioners for Finite Element Approximations
of Convection-Diffusion Equations on structured meshes \thanks{The
work of the first author was partially supported by MIUR, grant
number and the work of the second author and third author was
partially supported by MIUR, grant number 2006017542 and
20083KLJEZ.}}
\author{Alessandro Russo\thanks{Dipartimento di Matematica e Applicazioni,
Universit\`a di Milano Bicocca, via Cozzi 53, 20125 Milano, Italy
({\tt alessandro.russo@unimib.it}).} \and Stefano Serra
Capizzano\thanks{Dipartimento di Scienza ed Alta Tecnologia,
Universit\`a
    dell'Insubria - Sede di Como, Via Valleggio 11, 22100 Como,
    Italy ({\tt stefano.serrac@uninsubria.it}).}
        \and Cristina Tablino Possio\thanks{Dipartimento di Matematica e Applicazioni,
Universit\`a di Milano Bicocca, via Cozzi 53, 20125 Milano, Italy
({\tt cristina.tablinopossio@unimib.it}).}}
\begin{document}
\maketitle
\begin{abstract}
The paper is devoted to the spectral analysis of effective preconditioners for linear systems obtained via a
 Finite Element approximation to diffusion-dominated convection-diffusion equations. We consider a model setting in which
the structured finite element partition is  made by equi-lateral
triangles. Under such assumptions,
 if the problem is coercive, and the diffusive and convective coefficients are regular enough, then the proposed preconditioned
matrix sequences exhibit a strong clustering at unity, the preconditioning matrix sequence and the original matrix sequence
are spectrally equivalent, and the eigenvector matrices have a mild conditioning. The obtained results allow to show the optimality
of the related preconditioned Krylov methods.
The interest of such a study relies on the observation that
automatic grid generators tend to construct equi-lateral triangles
when the mesh is fine enough. Numerical tests, both on the model
setting and in the non-structured case, show the effectiveness of
the proposal and the correctness of the theoretical findings.
\end{abstract}
\pagestyle{myheadings} \thispagestyle{plain} \markboth{A. RUSSO,
S. SERRA CAPIZZANO, AND C. TABLINO POSSIO}{FE Preconditioning of
convection-diffusion Eqs}
%
\section{Introduction} \label{sez:introduction}
The paper is concerned with the spectral and computational analysis of
effective preconditioners for linear systems arising from Finite Element approximations to the elliptic convection-diffusion problem
\begin{equation} \label{eq:modello}
\left \{
\begin{array}{l}
\mathrm{div} \left(-\a(\mathbf{x}) \nabla u +\b(\mathbf{x}) u
\right)
=f, \quad \mathbf{x}\in \Omega,\\
u_{|\partial \Omega}=0,
\end{array}
\right.
\end{equation}
with $\Omega$ domain of $\mathbb{R}^2$.
%
We consider a model setting in which the structured finite element
partition is  made by equi-lateral triangles. The interest of such
a partition relies on the observation that automatic grid
generators tend to construct equi-lateral triangles when the mesh
is fine enough. \par
The analysis is performed having in mind two popular
preconditioned Krylov methods. More precisely, we analyze the
performances of the Preconditioned Conjugate Gradient (PCG) method
in the case of the diffusion problem and of the Preconditioned
Generalized Minimal Residual (PGMRES) in the case of the
convection-diffusion problem. \par
We define the preconditioner as a combination of a basic
(projected) Toeplitz matrix times diagonal structures. The
diagonal part takes into account the variable coefficients in the
operator of (\ref{eq:modello}), and especially the diffusion
coefficient $\a(\mathbf{x})$, while the (projected) Toeplitz part
derives from a special approximation of (\ref{eq:modello}) when
setting the diffusion coefficient to $1$ and the  convective
velocity field to $0$. Under such assumptions,
 if the problem is coercive, and the diffusive and convective coefficients are regular enough, then the proposed preconditioned
matrix sequences have a strong clustering at unity, the preconditioning matrix sequence and the original matrix sequence
are spectrally equivalent, and the eigenvector matrices have a mild conditioning. The obtained results allow to show the optimality
of the related preconditioned Krylov methods. It is important to stress that interest of such a study relies on the observation that
automatic grid generators tend to construct equi-lateral triangles when the mesh is fine enough.
Numerical tests, both on the model setting and in the non-structured case, show the effectiveness of the proposal and the correctness
of the theoretical findings.
%
\par
The outline of the paper is as follows. In Section \ref{sez:fem} we report a brief description of
the FE approximation of convection-diffusion equations and the preconditioner definition.
Section \ref{sez:clustering} is
devoted to the spectral analysis of the  underlying preconditioned matrix sequences,
in the case of structured uniform meshes. In Section \ref{sez:numerical_tests}, after a preliminary discussion
on complexity issues, selected numerical tests illustrate the convergence properties
stated in the former section and their extension
under weakened assumption or in the case of unstructured meshes.
A final Section \ref{sez:conclusions} deals with perspectives and future works.
\section{Finite Element approximation and Preconditioning Strategy} \label{sez:fem}
Problem (\ref{eq:modello}) can be stated in variational form as follows:
\begin{equation} \label{eq:formulazione_variazionale}
\left \{
\begin{array}{l}
\textrm{find $u \in H_0^1(\Omega)$ such that} \\
\int_\Omega \left ( \a \nabla u \cdot \nabla \varphi -\b \cdot \nabla
\varphi \ u
\right )
=\int_\Omega f \varphi  \quad \textrm{for all } \varphi \in  H_0^1(\Omega),
\end{array}
\right.
\end{equation}
where $H_0^1(\Omega)$ is the space of square integrable functions, with square integrable weak derivatives vanishing on
$\partial \Omega$.
We assume that $\Omega$ is a polygonal domain and we make the following hypotheses on the coefficients:
\begin{equation} \label{eq:ipotesi_coefficienti}
\left \{
\begin{array}{l}
\a \in {\bf C}^2(\overline \Omega),\quad \textrm{ with } \a(\mathbf{x}) \ge a_0 >0, \\
\b \in {\bf C}^1(\overline \Omega),\quad \textrm{ with } \mathrm{div} (\b) \ge 0 \textrm{ pointwise in } \Omega, \\
f \in {L}^2(\Omega).
\end{array}
\right.
\end{equation}
%
%
The previous assumptions guarantee existence and uniqueness for problem (\ref{eq:formulazione_variazionale})
and hence the existence and uniqueness of the (weak) solution for problem (\ref{eq:modello}).
\par
For the sake of simplicity, we restrict ourselves to linear finite element approximation of problem
(\ref{eq:formulazione_variazionale}). To this end, let $\mathcal{T}_h=\{K\}$ be a usual finite element partition
of $\overline \Omega$ into triangles, with $h_K=\mathrm{diam}(K)$ and $h=\max_K{h_K}$.
Let $V_h \subset H^1_0(\Omega)$ be the space of linear finite elements, i.e.
\[
V_h=\{\varphi_h : \overline \Omega \rightarrow \mathbb{R} \ \textrm{ s.t. } \varphi_h \textrm{ is continuous,  }
{\varphi_h}_{|_K} \textrm{ is linear, and }
{\varphi_h}_{| \partial \Omega} =0
 \}.
\]
The finite element approximation of problem (\ref{eq:formulazione_variazionale}) reads:
\begin{equation} \label{eq:formulazione_variazionale_fe}
\left \{
\begin{array}{l}
\textrm{find $u_h \in V_h$ such that} \\
\int_\Omega \left (\a \nabla u_h \cdot \nabla \varphi_h -\b \cdot \nabla
\varphi_h \ u_h
\right )
=\int_\Omega f \varphi_h \quad \textrm{for all } \varphi_h \in  V_h.
\end{array}
\right.
\end{equation}
For each internal node $i$ of the mesh $\mathcal{T}_h$, let $\varphi_i \in V_h$ be such that $\varphi_i(\textrm{node }i)=1$, and
$\varphi_i(\textrm{node }j)=0$ if $i\ne j$. Then, the collection of all $\varphi_i$'s is a base for $V_h$. We will denote
by $n(h)$ the number of the internal nodes of $\mathcal{T}_h$, which corresponds to the dimension of $V_h$. Then, we write $u_h$ as
\(
u_h=\sum_{j=1}^{n(h)}u_j \varphi_j
\)
and the variational equation (\ref{eq:formulazione_variazionale_fe}) becomes an algebraic linear system:
\begin{equation} \label{eq:modello_discreto}
\sum_{j=1}^{n(h)}\left (\int_\Omega \a \nabla \varphi_j \cdot \nabla
\varphi_i - \b \cdot \nabla \varphi_i\ \varphi_j
\right ) u_j =\int_\Omega f \varphi_i, \quad i=1,\ldots, n(h).
\end{equation}
According to these notations and definitions,
the algebraic equations in (\ref{eq:modello_discreto}) can be rewritten in
matrix form as the linear system
\begin{equation} \label{eq:def_A}
A_n(\a,\b) \mathbf{x} = \mathbf{b}, \quad  A_n(\a,\b)=\Theta_n(\a)+\Psi_n(\b)\in \mathbb{R}^{n\times n}, \
n=n(h),
\end{equation}
where $\Theta_n(\a)$ and $\Psi_n(\b)$ represent the approximation of the diffusive
term and approximation of the convective term, respectively. More precisely, we have
\begin{eqnarray}
(\Theta_n(\a))_{i,j} &=& \int_\Omega \a \ \nabla \varphi_j \cdot \nabla
\varphi_i, \label{eq:def_theta}\\
(\Psi_n(\b))_{i,j}&=& -\int_\Omega (\b \cdot \nabla \varphi_i ) \
\varphi_j, \label{eq:def_psi}
\end{eqnarray}
where suitable quadrature formula are considered in the case of
non constant coefficient functions $\a$ and $\b$.\par
As well known, the main drawback in the linear system resolution is due to the asymptotical ill-conditioning
(i.e. very large for large dimensions),
so that preconditioning is highly recommended. Hereafter, we refer to a preconditioning strategy
previously analyzed in the case of FD/FE
approximations of the diffusion problem \cite{S-NM-1999,ST-LAA-1999,ST-ETNA-2000,ST-ETNA-2003,ST-SIMAX-2003,ST-NA-2001}
and recently applied to FD/FE approximations \cite{BGST-NM-2005,BGS-SIMAX-2007,RT-TR-2008} of (\ref{eq:modello})
with respect to the Preconditioned Hermitian and Skew-Hermitian Splitting (PHSS) method \cite{BGN-SIMAX-2003,BGST-NM-2005}.
More precisely, the preconditioning matrix sequence $\{P_n(\a)\}$  is defined as
\begin{equation} \label{eq:def_P}
P_n(\a)=D_n^{\frac{1}{2}}(\a) A_n(1,0) D_n^{\frac{1}{2}}(\a)
\end{equation}
where
$D_n(\a)=\mathrm{diag}(A_n(\a,0))\mathrm{diag}^{-1}(A_n(1,0))
$, i.e., the suitable scaled main diagonal of $A_n(\a,0)$ and
clearly $A_n(\a,0)$ equals $\Theta_n(\a)$.\par
The computational aspects of this preconditioning strategy with
respect to Krylov methods will be discussed later in section
\ref{sez:complexity_issues}. Here, preliminarily we want to stress
as the preconditioner is tuned only with respect to the diffusion
matrix $\Theta_n(\a)$: in other words, we are implicity assuming
that the convection phenomenon is not dominant, and no
stabilization is required in order to avoid spurious oscillations
into the solution. \par
%
%
Moreover, the spectral analysis is performed in the non-Hermitian
case by referring to the Hermitian and skew-Hermitian (HSS)
decomposition of $A_n(\a,\b)$ (that can be performed on any single
elementary matrix related to $\mathcal{T}_h$ by considering the
standard assembling procedure).\par
%
%
%
%
According to the definition, the HSS decomposition is given by
\begin{eqnarray}
\mathrm{Re}(A_n(\a,\b)) &:=& \frac{A_n(\a,\b)+A_n^H(\a,\b)}{2} = \Theta_n(\a) +
\mathrm{Re}(\Psi_n(\b)), \label{eq:def_ReA}\\
\mathrm{i}\,  \mathrm{Im}(A_n(\a,\b)) &:=& \mathrm{i}\,\frac{A_n(\a,\b)-A_n^H(\a,\b)}{2\mathrm{i}} = \mathrm{i}\,
\mathrm{Im}(\Psi_n(\b)), \label{eq:def_ImA}
\end{eqnarray}
where
\begin{eqnarray}
\mathrm{Re}(\Psi_n(\b)) &=&
\frac{1}{2}(\Psi_n(\b)+\Psi_n^T(\b))=E_n(\b), \label{eq:def_E_nb}
%
\end{eqnarray}
since by definition, the diffusion term $\Theta_n(\a)$ is a
Hermitian matrix and does not contribute to the skew-Hermitian
part of $A_n(\a,\b)$. Notice also that $E_n(\b)=0$ if
$\mathrm{div}(\b)=0$.
More in general, the matrix $\mathrm{Re}(A_n(\a,\b))$ is symmetric and positive definite
whenever $\lambda_{\min}(\Theta_n(\a)) \ge \rho(E_n(\b))$. Indeed, without the condition
$\mathrm{div}(\b)\ge 0$, the matrix $E_n(\b)$ does not have a definite sign in general:
in fact, a specific analysis of the involved constants is required in order to guarantee the nonnegativity of
the term $E_n(\b)$.  Moreover,
the Lemma below allows to obtain further information regarding such a spectral assumption,
where $\|\cdot\|_2,\ \|\cdot\|_\infty$ indicate both the usual vector norms and the induced matrix
norms.\newline
\begin{lemma} \emph{\cite{RT-TR-2008}} \label{lemma:normaE} 
Let $\{E_n(\b)\}$ be the matrix sequence defined
according to \emph{(\ref{eq:def_E_nb})}.
%
%
Under the assumptions in \emph{(\ref{eq:ipotesi_coefficienti})}, then we find
\(
\|E_n(\b)\|_2  \le \|E_n(\b)\|_\infty \le Ch^2,
\)
with $C$ absolute positive constant only depending on $\b(\mathbf{x})$ and $\Omega$.
The claim holds both in the case in which the matrix elements in
\emph{(\ref{eq:def_psi})} are evaluated exactly and whenever a quadrature formula
with error $O(h^2)$ is considered for approximating the  involved integrals. \newline
\end{lemma}
\par Hereafter, we will denote by $\{A_n(\a,\b)\}$, $n=n(h)$, the
matrix sequence associated to a family of meshes
$\{\mathcal{T}_h\}$, with decreasing finesse parameter $h$. As
customary, the whole preconditioning analysis will refer to a
matrix sequence instead to a single matrix, since the goal is to
quantify the difficulty of the linear system resolution in
relation to the accuracy of the chosen approximation scheme.
\section{Spectral analysis and clustering properties in the case of structured uniform meshes} \label{sez:clustering}
The aim of this section is to analyze the spectral properties of
the preconditioned matrix sequences $\{P_n^{-1}(\a)(A_n(\a,\b))\}$
in the case of some special domains $\Omega$ partitioned with
structured uniform meshes, so that spectral tools derived from
Toeplitz theory \cite{BS-Springer-1998,S-LAA-1998} can be
successfully applied. The applicative interest of the considered
type of domains will be motivated in the short.
\par Indeed, let $f$ be a $2-$variate
Lebesgue integrable function defined over  $\mathcal{D}=(-\pi,
\pi]^2$. By referring to  the Fourier coefficients of this
function $f$, called generating function,
\[
a_j = \frac{1}{4\pi^2} \int_\mathcal{D} f(s)e^{-\mathrm{i}<j,\, s>} ds, \quad \mathrm{i}^2=-1, \ j=(j_1,j_2)\in \mathbb{Z}^2,
\]
with $<j,\, s> =j_1s_1+j_2s_2$, one can build the sequence of
Toeplitz matrices $\{T_n(f)\}$. The matrix $T_n(f) \in
\mathbb{C}^{n \times n}$ is said to be the Toeplitz matrix of
order $N=(N_1,N_2)$, $n=N_1N_2$, generated by $f$.\par Now, the
spectral properties of the matrix sequence $\{T_n(f)\}$
are completely understood and characterized in terms of the
underlying generating functions (see, e.g.,
\cite{BS-Springer-1998} for more details). In the following, with
respect to the latter claim, we will refer to the notion of
equivalent generating functions. In such a respect, we claim  that
two nonnegative function $f$ and $g$ defined over a domain
$\mathcal{D}$ are equivalent if and only if $f=O(g)$ and $g=O(f)$,
where $\alpha=O(\beta)$ means that there exists a pure positive
constant $c$ such that $\alpha \le c \beta$ almost everywhere on
$\mathcal{D}$. \par
\begin{figure}
\centering \epsfig{file=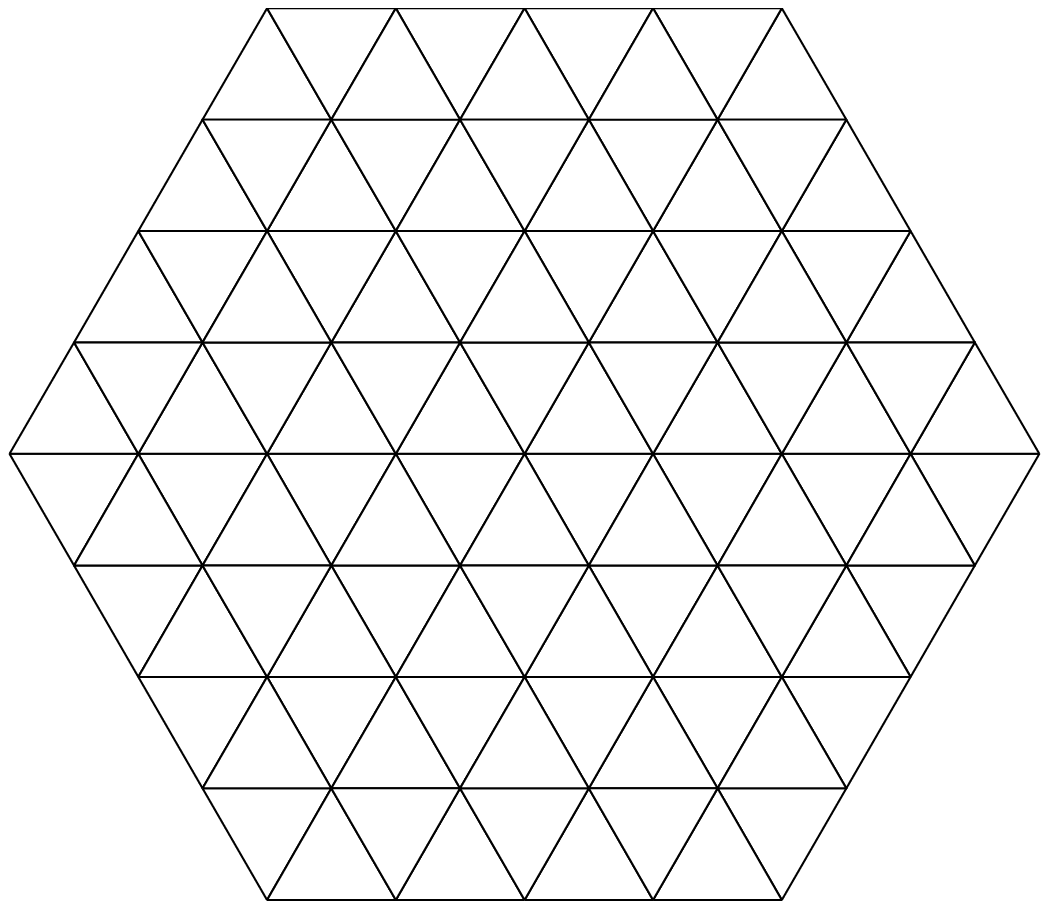,height=4cm} \
\epsfig{file=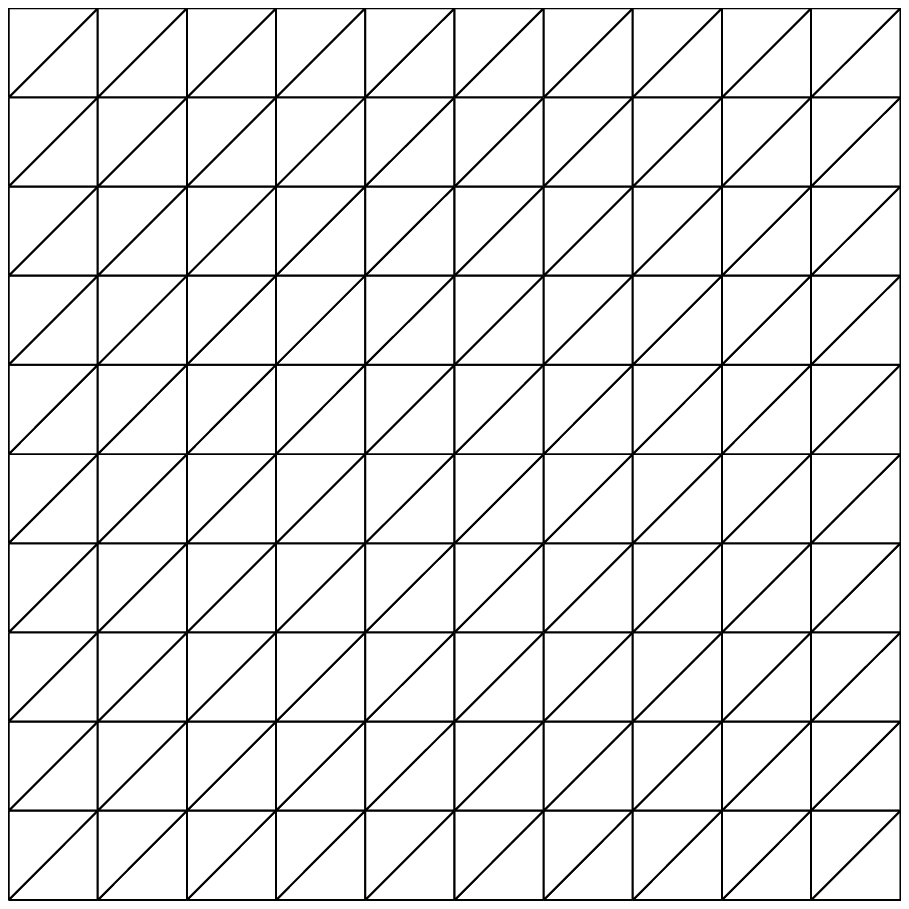,height=4cm}\\
(a)\hskip 5cm (b)
\caption{Uniform structured mesh with equilateral or rectangular
triangles.} \label{fig:esagono_esagoni}
\end{figure}
The main motivation of this paper lies in the analysis of the
template case reported in Figure \ref{fig:esagono_esagoni}.a,
where we consider a partition into equilateral triangles. It is
self-evident that such a problem represents just an academic
example. However, it is a fact that a professional mesh generator
will locally produce a partitioning, which is ``asymptotically''
similar to the one reported in the figure. \par Our goal is to
prove the PCG optimality with respect to the diffusion problem
approximation, i.e., the number of iterations for reaching the
solution, within a fixed accuracy, can be bounded from above by a
constant independent of the dimension $n=n(h)$. Some additional
results about PGMRES convergence properties in the case of the
convection-diffusion problem approximation are also reported, both
in the case of the hexagonal domain $\Omega$ with a structured
uniform mesh as in Figure \ref{fig:esagono_esagoni}.a and in the
case $\Omega=(0,1)^2$ with a structured uniform mesh as in Figure
\ref{fig:esagono_esagoni}.b, 
so extending previous results proved in \cite{RT-TR-2008} with
respect to the PHSS method \cite{BGST-NM-2005}. \par The spectral
analysis makes reference to the following definition. \par
\begin{definition} \emph{\cite{T-LAA-1996}} \label{def:cluster}
Let $\{A_n\}$ be a sequence of Hermitian matrices of increasing dimensions
$n$. The sequence $\{A_n\}$ is clustered at $p$ in the eigenvalue sense if for any $\varepsilon
>0$, $\# \left\{i \left.\right|\lambda_i(A_n) \notin (p - \varepsilon, p
+ \varepsilon)\right\}=o(n)$.
The sequence $\{A_n\}$ is properly (or strongly) clustered at $p$
if for any $\varepsilon> 0$ the number of the eigenvalues of $A_n$
not belonging to $(p - \varepsilon, p + \varepsilon)$ can be
bounded by a pure constant eventually depending on $\varepsilon$,
but not on $n$. In the case of non-Hermitian matrices, the same definition can be extended by
considering  the complex disk $D(p,\varepsilon)$, instead of the interval $(p - \varepsilon, p + \varepsilon)$.
%
\end{definition}
\subsection{Diffusion Equations} \label{sez:diffusion_equation} 
\par
We start by considering the simpler diffusion problem, i.e., the
analysis concerns the Hermitian matrix sequences $\{A_n(\a)\}$,
$A_n(\a)=\Theta_n(\a)$. \par Due to the very special choice of the
domain $\Omega$, the matrices $A_n(\a)=\Theta_n(\a)$ arising in
the case $\a\equiv 1$ with structured meshes as in Figure
\ref{fig:esagono_esagoni}.a result to be a proper projection of
Toeplitz matrices generated by the function
$\tilde{f}(s_1,s_2)=\sqrt{3}(6-2\cos(s_1)-2\cos(s_2)-2\cos(s_1+s_2)){/3}$,
$(s_1,s_2) \in \mathcal{D}=(-\pi, \pi]^2$, related to a bigger
parallelogram shaped domain $\Omega_N$ containing $\Omega$ (see
Figure \ref{fig:esagono_esagoni_parallelogrammi}), i.e., \(
A_n(1)= \Pi\, T_N(\tilde{f}) \, \Pi^T \), with $N$ number of the
internal nodes. Here, $\Pi \in \mathbb{R}^{n\times N}$, $n\le N$
is a proper projection matrix simply cutting those rows (columns)
referring to nodes belonging to $\Omega_N$, but not to $\Omega$.
Thus, the matrix is full rank, i.e., $\mathrm{rank}(\Pi)=n$, and
$\Pi \, \Pi^T=I_n$.  \par In the same way, we can also consider
the Toeplitz matrix with the same generating function arising when
considering a smaller parallelogram shaped domain
$\Omega_{\tilde{N}}$ contained in $\Omega$, i.e., \(
T_{\tilde{N}}(\tilde{f})  = \tilde{\Pi} \, A_n(1)  \,
\tilde{\Pi}^T \), with $\tilde{\Pi} \in
\mathbb{R}^{{\tilde{N}}\times n}$, ${\tilde{N}}\le n$  proper
projection matrix and such that
$\mathrm{rank}(\tilde{\Pi})={\tilde{N}}$, and $\tilde{\Pi} \,
\tilde{\Pi}^T=I_{\tilde{N}}$.  \par
%
%
These embedding arguments allow to bound, according to the min-max principle (see \cite{ST-SIMAX-2003}
for more details on the use of this proof technique), the minimal eigenvalue of the matrices $A_n(1)$ as follows
\begin{equation} \label{eq:lmin_proiezione}
\lambda_{\min} (T_{N}(\tilde{f})) \le \lambda_{\min} (A_n(1)) \le \lambda_{\min} (T_{\tilde{N}}(\tilde{f})) .
\end{equation}
\begin{figure}
\centering
\epsfig{file=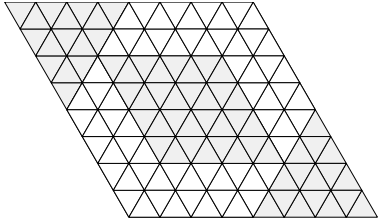,height=4cm}
\caption{Uniform structured mesh on $\Omega$, $\Omega_N$, $\Omega_{\tilde{N}}$.}
\label{fig:esagono_esagoni_parallelogrammi}
\end{figure}
A first technical step in our spectral analysis concerns
relationships between this generating function $\tilde{f}$ and the
more classical generating function
$f(s_1,s_2)=4-2\cos(s_1)-2\cos(s_2)$ arising in the case of FE
approximations on a square $\Omega=(0,1)^2$ with Friedrichs-Keller
meshes (see Figure \ref{fig:esagono_esagoni}.b), or standard FD
discretizations.\par It can be easily observed that these two
function are equivalent, in the sense of the previously reported
definition, since we have
\[
\frac{\sqrt{3}}{3} \, f \le \tilde{f} \le \sqrt{3}\, f \quad \textrm{on } \mathcal{D}=(-\pi, \pi]^2.
\]
Thus, because $T_n(\cdot)$ is a matrix-valued linear positive
operator (LPO) for every $n$ \cite{Serra-CAMA-2000}, the Toeplitz
matrix sequences generated by this equivalent functions result to
be spectrally equivalent, i.e., for any $n$
\begin{equation} \label{eq:rel_toeplitz}
\frac{\sqrt{3}}{3}\, T_n(f) \le T_n(\tilde{f}) \le \sqrt{3} \, T_n(f);
\end{equation}
furthermore, due to the strict positivity of every $T_n(\cdot)$
(see \cite{Serra-CAMA-2000} for the precise definition), since $
\tilde{f} -\frac{\sqrt{3}}{3}$ is not identically zero, we have
that
\[
 T_n(\tilde{f}) -\frac{\sqrt{3}}{3}\, T_n(f)
\]
is positive definite, and since $\sqrt{3}\, f - \tilde{f}$ is not
identically zero, we find that
\[
\sqrt{3} \, T_n(f)- T_n(\tilde{f})
\]
is also positive definite. Here, we are referring to the standard
ordering relation between Hermitian matrices, i.e., the notation
$X \ge Y$, with $X$ and $Y$ Hermitian matrices, means that $X - Y$
is nonnegative definite. \par An interesting remark pertains to
the fact that the function $\tilde{f}$ is the most natural one
from the FE point of view, since no contribution are lost owing to
the gradient orthogonality as instead in the case related to $f$;
nevertheless its relationships with the function $f$ can be fully
exploited in performing the spectral analysis. More precisely,
from (\ref{eq:lmin_proiezione}) and (\ref{eq:rel_toeplitz}) and
taking into account (see e.g.  \cite{S-LAA-1998}) that
\[
 \lambda_{\min}(T_n(f)) =8\sin^2\left({\pi\over 2}h\right),\ \  h={1\over n+1},
\]
we deduce
\begin{equation} \label{eq:lmin_toeplitz_bottom}
\lambda_{\min}(A_n(1)) \ge \lambda_{\min}(T_N(\tilde{f})) \ge \frac{\sqrt{3}}{3} \, \lambda_{\min}(T_N(f))=
\frac{\sqrt{3}}{3}8\sin^2\left({\pi\over 2}h_N\right) \sim h_N^2.
\end{equation}
Following the very same reasoning, we also find that
\begin{equation} \label{eq:lmin_toeplitz_top}
\lambda_{\min}(A_n(1)) \le \lambda_{\min}(T_{\tilde{N}}(\tilde{f})) \le {\sqrt{3}} \, \lambda_{\min}(T_{\tilde{N}}(f))
= 8 {\sqrt{3}} \sin^2\left({\pi\over 2}h_{\tilde{N}}\right)\sim  h_{\tilde{N}}^2.
\end{equation}
Since, by the embedding argument, both $N$ and $\tilde N$ are asymptotic to $n$, it follows that
$\lambda_{\min}(A_n(1)) \sim h^2$.
%
%
%
Finally, following the same analysis for the maximal eigenvalue, we find
\[
\frac{\sqrt{3}}{3} \, \lambda_{\max}(T_N(f)) \le \lambda_{\max}(A_n(1)) \le  {\sqrt{3}} \, \lambda_{\max}(T_{\tilde{N}}(f)),
\]
where, by  \cite{BS-Springer-1998}, we know that
\[
\lim_{n\rightarrow\infty}  \lambda_{\max}(T_n(f)) =\max_{\mathcal{D}} f=8,
\]
and hence the spectral condition number of $A_n(1)$ grows as $h^{-2}$ i.e.
\begin{equation} \label{eq:cond-an1}
K_2(A_n(1)) \sim  h^{-2},\ \ \ A_n(1)=\Theta_n(1),
\end{equation}
where the constant hidden in the previous relation is mild and can be easily estimated.

It is worth stressing  that the same matrix $\Pi$ can also be
considered in a more general setting. In fact, the matrix sequence
$\{A_n(\a)\}$ can also be defined as $\{A_n(\a)\}= \{\Pi\, A_N(\a)
\, \Pi^T\}$, since again each internal node in $\Omega$ is a
vertex of the same constant number of triangles. Therefore, by
referring to projection arguments, the spectral analysis can be
equivalently performed both on the matrix sequence $\{A_n(\a)\}$
and $\{A_N(\a)\}$. \par No matter about this choice, we make use
of a second technical step, which is based on standard Taylor's
expansions.
\par
\begin{lemma}\label{lemma:Taylor_expansion_esagoni}
Let  $\a \in {\bf C}^2(\overline \Omega)$,  with $\a(\mathbf{x})
\ge a_0 >0$ and $\Omega$ hexagonal domain partitioned as in Figure
\ref{fig:esagono_esagoni}.a. For any $p$ such that
$(\Theta_n(\a))_{s,s-p}\ne 0$ there exists a proper
$(x^*_p,y^*_p)$  such that the Taylor's expansions centered at a
proper $(x^*_p,y^*_p)$ have the form
\begin{eqnarray}
(\Theta_n(\a))_{s,s-p}  & = & a(x^*_p,y^*_p) (A_n(1))_{s-p} +  h^2 D_p + o(h^2), \label{eq:exp.a}\\
(\Theta_n(\a))_{s,s} & = & a(x^*_p,y^*_p) (A_n(1))_{s,s}+h B_p  +h^2 C_p +o(h^2), \label{eq:exp.b}\\
(\Theta_n(\a))_{s-p,s-p} & = &  a(x^*_p,y^*_p) (A_n(1))_{s,s}-h
B_p
         +h^2  C_p +o(h^2), \label{eq:exp.c}
\end{eqnarray}
where $B_p, C_p$ and $D_p$ are constants independent of $h$.
\end{lemma}
\begin{proof}
\begin{figure}
\centering 
\epsfig{file=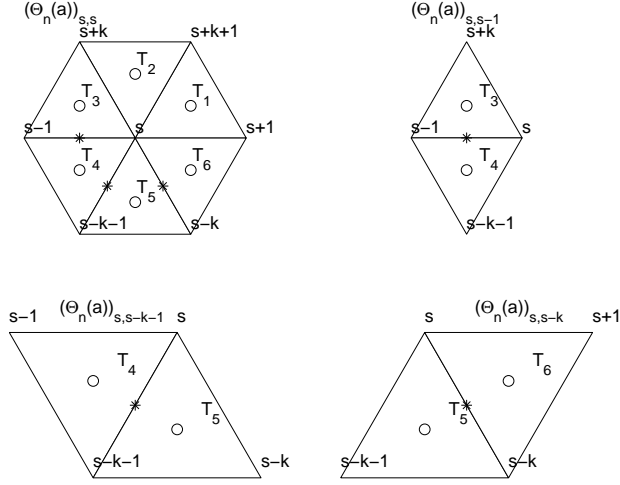,clip=, width=10cm} \\
\vskip -1.0cm \caption{Centers of Taylor's expansion.}
\label{fig:simmetrie}
\end{figure}
In the cases at hand, the validity of this claim is just a direct
check: the key point relies in the symmetry properties induced by
the structured uniform nature of the considered meshes, both in
the case of $\Omega$ and $\Omega_N$. Hereafter, we give some
detail with respect to the case of the hexagonal domain
partitioned as in Figure \ref{fig:esagono_esagoni}.a,
see \cite{ST-NA-2001} for the proof in the case $\Omega=(0,1)^2$.
Thus, let's consider the Taylor's expansion centered at a proper
$(x^*_p,y^*_p)$. By calling $a^*=a(x^*_p,y^*_p)$ and similarly
denoting the derivatives of the diffusion coefficient, it holds
that
\begin{eqnarray*}
\int_K a(x,y) &=& \int_K a^*+(x-x^*) a^*_x+(y-y^*) a^*_y \\
&& \quad +\frac{1}{2} (x-x^*)^2 a^*_{xx} +\frac{1}{2} (y-y^*)^2
a^*_{yy} +(x-x^*)(y-y^*)
a^*_{xy}+o(\tau^2) \\
&=& |K| (a^*+(x-x^*)_{|b_K} a^*_x+(y-y^*)_{|b_K} a^*_y )\\
&& \quad  + \int_K \frac{1}{2} (x-x^*)^2 a^*_{xx} +\frac{1}{2}
(y-y^*)^2 a^*_{yy} +(x-x^*)(y-y^*) a^*_{xy}+o(\tau^2) \\
&=& |K| (a^*+(x-x^*)_{|b_K} a^*_x+(y-y^*)_{|b_K} a^*_y \\
&& \quad  +  \frac{1}{2} (x-x^*)^2_{|b_K} a^*_{xx} +\frac{1}{2}
(y-y^*)^2_{|b_K} a^*_{yy} +(x-x^*)(y-y^*)_{|b_K}
a^*_{xy})+o(\tau^2)
\end{eqnarray*}
with $\tau=\| [ x-x^,y-y^* ]^T \|$ and $b_K$ denoting the
barycenter of the triangle $K$. By referring to the assembling
procedure, we will choose as center of the related Taylor's
expansion in the case of Figure
\ref{fig:simmetrie}.b-\ref{fig:simmetrie}.d the point marked by
$*$. The symmetric position of the involved triangle barycenters
(marked by $\circ$), with respect to $*$ allows to claim
(\ref{eq:exp.a}). The same holds true for (\ref{eq:exp.b}) by
considering each pertinent point $*$ in Figure
\ref{fig:simmetrie}.a and, analogously, for (\ref{eq:exp.c}).
\end{proof}
%
\begin{lemma}\label{lemma:Taylor_expansion} \emph{\cite{ST-NA-2001}}
%
%
Let's considering the following notation of scaled matrix
$X^*=D_n^{-\frac{1}{2}}(a)XD_n^{-\frac{1}{2}}(a)$ of a given
matrix $X$. Under the assumptions of \emph{Lemma
\ref{lemma:Taylor_expansion_esagoni}} the following representation
holds true
\begin{equation}\label{quadratico}
\Theta_n^*(\a)=D_{n}^{-1/2}(\a)\Theta_n(\a) D_{n}^{-1/2}(\a)
=\Theta_n(1) +h^2 F_n(\a)+o(h^2)G_n(\a)
\end{equation}
where $F_n$ and $G_n$ are uniformly bounded in spectral norm and
have the same pattern as $\Theta_n^*(\a)$.
\par
%
%
\end{lemma}
We are now ready to prove the PCG optimality.
\begin{theorem} \label{teo:clu+se_A_esagoni}
Let $\{A_n(\a)\}$ and $\{P_n(\a)\}$ be the Hermitian positive
definite matrix sequences defined according to
\emph{(\ref{eq:def_A})} and \emph{(\ref{eq:def_P})} in the case of
the hexagonal domain partitioned as in Figure
\ref{fig:esagono_esagoni}.a.
%
%
Under the assumptions in \emph{(\ref{eq:ipotesi_coefficienti})},
the sequence $\{P_n^{-1}(\a) A_n(\a)\}$ is
properly clustered at $1$.
Moreover, for any $n$ all the eigenvalues of $P_n^{-1}(\a)
A_n(\a)$ belong to an interval $[d,D]$ well
separated from zero {\em [Spectral equivalence property]}.
\end{theorem}
\begin{proof} 
Since $A_n(\a)= \Pi\, A_N(\a) \, \Pi^T$, with
$\Pi \in \mathbb{R}^{n\times N}$  such that  $\mathrm{rank}(\Pi)=n$, and $\Pi \, \Pi^T=I_n$, we have
\begin{eqnarray*}
P_n^{-1}(\a) A_n(\a) & = &
  [D_n^{\frac{1}{2}}(\a) A_n(1,0) D_n^{\frac{1}{2}}(\a)]^{-1}  A_n(\a) \\
& = &
  [\Pi \, D_N^{\frac{1}{2}}(\a) \, \Pi^T \, \Pi A_N(1) \, \Pi^T \, \Pi \, D_N^{\frac{1}{2}}(\a) \, \Pi^T ]^{-1} \Pi \, A_N(\a) \, \Pi^T \\
& = &
  [\Pi \, \Pi^T \, \Pi \, D_N^{\frac{1}{2}}(\a) A_N(1)  D_N^{\frac{1}{2}}(\a) \, \Pi^T \, \Pi \, \Pi^T ]^{-1} \Pi \, A_N(\a) \, \Pi^T \\
& = &
  [\Pi \, D_N^{\frac{1}{2}}(\a) A_N(1)  D_N^{\frac{1}{2}}(\a) \, \Pi^T  ]^{-1} \Pi \, A_N(\a) \, \Pi^T \\
& = &
  [\Pi P_N^{-1}(\a) \, \Pi^T  ]^{-1} \Pi \, A_N(\a) \, \Pi^T.
\end{eqnarray*}
Since $\Pi$ is full rank, it is evident that the spectral
behavior of
\[
[ \Pi  P_N(\a) \Pi^T]^{-1}\Pi A_N(\a) \Pi^T
\]
is in principle better than the one of $P_N^{-1}(\a)A_N(\a)$, to which we can address the spectral analysis.
Thus, the proof technique refers to \cite{ST-NA-2001,ST-SIMAX-2003} and the very key step is given by the relations outlined in
(\ref{eq:lmin_toeplitz_bottom}) and (\ref{eq:lmin_toeplitz_top}).
\end{proof}
\par
%
It is worth stressing  that the previous claims can also be proved
by considering the sequence $\{P_n^{-1}(\a) A_n(\a)\}$, instead of
$\{P_N^{-1}(\a) A_N(\a)\}$, simply by directly referring to the
quoted asymptotical expansions. The interest may concern the
analysis of the same uniform mesh on more general domains. \par
%
Moreover, the same spectral properties has been proved in the case of uniform structured meshes as in Figure \ref{fig:esagono_esagoni}.b
in \cite{ST-NA-2001}.
%
\subsection{Convection-Diffusion Equations} 
The natural extension of the claim in Theorem
\ref{teo:clu+se_A_esagoni}, in the case of the matrix sequence
$\{\mathrm{Re}(A_n(\a,\b))\}$ with $\mathrm{Re}(A_n(\a,\b)) \ne
\Theta_n(\a)$, can be proved under the additional assumptions of
Lemma \ref{lemma:normaE}, in perfect agreement with Theorem 5.3 in
\cite{RT-TR-2008}. \par
\begin{theorem} \label{teo:clu+se_ReA_esagoni}
Let $\{\mathrm{Re}(A_n(\a,\b))\}$ and $\{P_n(\a)\}$ be the
Hermitian positive definite matrix sequences defined according to
\emph{(\ref{eq:def_ReA})} and \emph{(\ref{eq:def_P})} in the case
of the hexagonal domain partitioned as in Figure
\ref{fig:esagono_esagoni}.a.
%
%
Under the assumptions in \emph{(\ref{eq:ipotesi_coefficienti})},
the sequence $\{P_n^{-1}(\a) \mathrm{Re}(A_n(\a,\b))\}$ is
properly clustered at $1$.
Moreover, for any $n$ all the eigenvalues of $P_n^{-1}(\a)
\mathrm{Re}(A_n(\a,\b))$ belong to an interval $[d,D]$ well
separated from zero {\em [Spectral equivalence property]}.\par The
claim holds both in the case in which the matrix elements in
\emph{(\ref{eq:def_theta})} and \emph{(\ref{eq:def_psi})} are
evaluated exactly and whenever a quadrature formula with error
$O(h^2)$ is considered to approximate the involved integrals.
\par
\end{theorem}
\begin{proof} 
%
The proof can be done verbatim as in \cite{RT-TR-2008}, since the key point is proved by referring to relations
in (\ref{eq:lmin_toeplitz_bottom}) and (\ref{eq:lmin_toeplitz_top}).
\end{proof}
\par
\begin{theorem} 
\label{teo:clu+sb_ImA}
Let $\{\mathrm{Im}(A_n(\a,\b))\}$ and $\{P_n(\a)\}$ be the
Hermitian matrix sequences defined according to
\emph{(\ref{eq:def_ImA})} and \emph{(\ref{eq:def_P})} in the case
of the hexagonal domain partitioned as in Figure
\ref{fig:esagono_esagoni}.a. \newline 
%
%
Under the assumptions in \emph{(\ref{eq:ipotesi_coefficienti})},
the sequence $\{P_n^{-1}(\a) \mathrm{Im}(A_n(\a,\b))\}$ is
spectrally bounded and properly clustered at $0$ with respect to
the eigenvalues. The claim holds both in the case in which the
matrix elements in \emph{(\ref{eq:def_psi})} are evaluated exactly
and whenever a quadrature formula with error $O(h^2)$ is
considered to approximate the involved integrals. \par
\end{theorem}
\begin{proof} 
This result has been proved in \cite{RT-TR-2008} with respect to Friedrichs-Keller triangulations, but it can easily be extended
to the case of matrix sequences arising in the case of a FE partitioning as in Figure \ref{fig:esagono_esagoni}.a by using the same arguments.
\end{proof}
\par
On the basis of these two splitted spectral results, we can easily
obtain the spectral description of the whole preconditioned matrix
sequence $\{P_n^{-1}(\a) A_n(\a,\b)\}$, according to the theorem
below. The proof technique refers to an analogous result proved in
\cite{BGS-SIMAX-2007} with respect to FD discretizations of
convection-diffusion equations.
\par
\begin{theorem}  \label{teo:cluA}
Let $\{(A_n(\a,\b))\}$ and $\{P_n(\a)\}$ be the matrix sequences
defined according to \emph{(\ref{eq:def_A})} and
\emph{(\ref{eq:def_P})}, both in the case of the hexagonal domain
$\Omega$ with a structured uniform mesh as in Figure
\ref{fig:esagono_esagoni}.a and in the case $\Omega=(0,1)^2$ with
a structured uniform mesh as in Figure \ref{fig:esagono_esagoni}.b
\par Under the assumptions in
\emph{(\ref{eq:ipotesi_coefficienti})}, the sequence
$\{P_n^{-1}(\a) A_n(\a,\b)\}$ is properly clustered at $1\in
\mathbf{C}^+$ with respect to the eigenvalues. In addition, these
eigenvalues all belong to a uniformly bounded rectangle with
positive real part, well separated from zero.\par The claim holds
both in the case in which the matrix elements in
\emph{(\ref{eq:def_theta})} and \emph{(\ref{eq:def_psi})} are
evaluated exactly and whenever a quadrature formula with error
$O(h^2)$ is considered to approximate the involved integrals.
\par
\end{theorem}
\begin{proof} 
A localization result for the eigenvalues of the sequence $\{P_n^{-1}(\a) (A_n(\a,\b))\}$
can be stated simply by referring to the properties of the
field of values of the preconditioned matrix: in fact, for any $i$
\[
\lambda_i(P_n^{-1}(\a)A_n(\a,\b) \in \mathcal{F}=\mathcal{F}\left(P_n^{-\frac{1}{2}}(\a)A_n(\a,\b)P_n^{-\frac{1}{2}}(\a)\right)
\]
where
\[
\mathcal{F}
= \left \{ z \in \mathbb{C} :
z=  \frac{x^H \mathrm{Re}(A_n(\a,\b)) x}{x^HP_n(\a)x} +\mathrm{i}\frac{x^H \mathrm{Im}(A_n(\a,\b)) x}{x^HP_n(\a)x},\ x \in \mathbf{C}^n\backslash \{0\}
\right \}.
\]
Since for any $i$
\begin{eqnarray*}
\lambda_i(P_n^{-1}(\a)\mathrm{Re}(A_n(\a,\b)) &\in&  
\left \{ z \in \mathbb{C} : z=  \frac{x^H \mathrm{Re}(A_n(\a,\b)) x}{x^HP_n(\a)x},\ x \in \mathbf{C}^n\backslash \{0\} \right \}, \\
\lambda_i(P_n^{-1}(\a)\mathrm{Im}(A_n(\a,\b)) &\in&  
\left \{ z \in \mathbb{C} : z=  \frac{x^H \mathrm{Im}(A_n(\a,\b)) x}{x^HP_n(\a)x},\ x \in \mathbf{C}^n\backslash \{0\} \right \}.
\end{eqnarray*}
the claimed localization result is obtained as a consequence of those previously proved in Theorems \ref{teo:clu+se_ReA_esagoni} and \ref{teo:clu+sb_ImA}.
Moreover, the same localization results together with the claimed clustering properties allow to apply Theorem 4.3 in \cite{BGS-SIMAX-2007}
and the proof is concluded.
\end{proof}
\par The subsequent result, giving estimates on the condition number
of the eigenvector matrix of the preconditioned structure, is of
interest in the study of the convergence speed of the GMRES. In
particular a logarithmic growth in number of iteration is
associated to a polynomial bound in the spectral condition number:
such a bound is precisely given below.
\begin{theorem}
\label{lemmma:eigenvectors} Let $\a(\mathbf{x})$ and
$\b(\mathbf{x})$ constant functions and let $\{(A_n(\a,\b))\}$ and
$\{P_n(\a)\}$ be the matrix sequences defined as in Theorem
\ref{teo:cluA}
\par $Then,
P_n(\a)^{-1}A_n(\a,\b)$ can be diagonalized as
\[
P_n(\a)^{-1}A_n(\a,\b)= V_n D_n V_n^{-1}
\]
where the matrix of the eigenvectors $V_n$ can be chosen such that
$K_2(V_n) \sim h^{-1}$.
\end{theorem}
\begin{proof}
Under the quoted assumption it holds that $P_n(\a)=\a \Theta_n(1)$
and $\Psi_n(\b)$ is skew-symmetric. Thus, we have
\begin{eqnarray*}
P_n(\a)^{-1}A_n(\a,\b) &=& I_n + \a^{-1}
\Theta_n^{-1}(1)\Psi_n(\b) \\
    &=& I_n +\Theta_n^{-\frac{1}{2}}(1) W_n
    \Theta_n^{\frac{1}{2}}(1),
\end{eqnarray*}
with $W_n= a^{-\frac{1}{2}}\Theta_n^{-\frac{1}{2}}(1) \Psi_n(\b)
\Theta_n^{-\frac{1}{2}}(1) a^{-\frac{1}{2}}$ skew-symmetric
matrix. Therefore, the matrix $I_n+W_n$ is a normal normal, so
that $I_n+W_n=Q_n D_n Q_n^H$, where $D_n$ is diagonal and $Q_n$ is
unitary. Consequently,
\begin{eqnarray*}
P_n(\a)^{-1}A_n(\a,\b) &=& \Theta_n^{-\frac{1}{2}}(1) (I_n+W_n)
    \Theta_n^{\frac{1}{2}}(1)\\
     &=& \Theta_n^{-\frac{1}{2}}(1) Q_n D_n Q_n^H \Theta_n^{\frac{1}{2}}(1)\\
     &=& V_n D_n  V_n^{-1}
\end{eqnarray*}
with $V_n=\Theta_n^{-\frac{1}{2}}(1) Q_n$ eigenvector matrix.
Due to the fact that $Q_n$ is unitary, we have $K_2(V_n)=K_2(\Theta_n^{-\frac{1}{2}}(1) Q_n)=K_2(\Theta_n^{-\frac{1}{2}}(1))$
and finally $K_2(V_n)\sim  h^{-1}$ by invoking the key relation (\ref{eq:cond-an1}).
\end{proof}
\ \par \
\section{Numerical tests} \label{sez:numerical_tests}

The section is divided into two parts. In the first we briefly
discuss the complexity features of our preconditioning proposals.
In the second part we report and critically discuss few numerical
experiments in which the meshes are both structured and
unstructured, while the regularity features required by the
theoretical analysis are somehow relaxed.

\subsection{Complexity issues} \label{sez:complexity_issues}
First of all, we report some remarks about the computational costs
of the proposed iterative procedure. \par The main idea is that
such a technique is easily applicable. In fact, a Krylov method is
considered, so simply requiring a matrix vector routine for sparse
matrices and a solver for the chosen preconditioner. \\ Therefore,
since the preconditioner is defined as $P_n(\a)=D_n^{{1}/{2}}(\a)
A_n(1,0) D_n^{{1}/{2}}(\a)$, where
$D_n(\a)=\mathrm{diag}(A_n(\a,0))\mathrm{diag}\!^{-1}(A_n(1,0))$,
the solution of the linear system in (\ref{eq:modello_discreto})
with matrix $A_n(\a,\b)$ is reduced to computations involving
diagonals and the matrix $A_n(1,0)$ ($A_n(\a)$ and $A_n(1,0)$ for
the diffusion problem, respectively). \par
%
%
As well known, whenever the domain is partitioned by considering a
uniform structured mesh this latter task can be efficiently
performed by means of fast Poisson solvers, among which we can
list those based on the cyclic reduction idea (see e.g.
\cite{BDGG-SINUM-1971,D-SIAMRev-1970,S-SIAMRev-1977}) and several
specialized algebraic or geometric multigrid methods (see e.g.
\cite{H-Springer-1985,TOS-AP-2001,S-NM-2002}). In addition, the
latter can be efficiently considered also in more general mesh
settings. The underlying idea is that the main effort in devising
efficient algorithms must be devoted only to this simpler problem
with constant coefficient, instead of the general one. \par Now,
the effectiveness of the proposed method is measured by referring
to the optimality definition below.
\par
\begin{definition}\emph{\cite{AN-1993}} \label{def:opt}
Let $\{A_m \mathbf{x}_m = \mathbf{b}_m\}$ be a given sequence of
linear systems of increasing dimensions. An iterative method is
  \emph{optimal} if
  \begin{enumerate}
    \item[{\rm 1.}] the arithmetic cost of each iteration is at most
      proportional to the complexity of a matrix-vector product with
      matrix $A_m$,
    \item[{\rm 2.}] the number of iterations for reaching the solution within a
      fixed accuracy can be bounded from above by a constant
      independent of $m$.
  \end{enumerate}
\end{definition}
\par
In such a respect Theorem \ref{teo:clu+se_A_esagoni} proves the
optimality of the PCG method: the iterations number for reaching
the solution within a fixed accuracy can be bounded from above by
a constant independent of the dimension $n=n(h)$ and the
arithmetic cost of each iteration is at most proportional to the
complexity of a matrix-vector product with matrix $A_n(\a,0)$.
\par
Moreover, Theorem \ref{teo:cluA} proves that all the eigenvalues
of the preconditioned matrix belong in a complex rectangle $\{z
\in \mathbb{C}: \mathrm{Re}(z) \in [d,D], \ \mathrm{Im}(z) \in
[-\hat{d},\hat{d}]\}$, with $d,D>0$, $\hat{d}\ge 0$ independent of
the dimension $n=n(h)$. It is worth stressing that the existence
of a proper eigenvalue cluster and the aforementioned localization
results in the preconditioned spectrum can be very important for
fast convergence of preconditioned GMRES iterations (see, e.g.,
\cite{BN-BIT-2003}).
\par
Finally, we want to give notice that the PCG/PGMRES numerical
performances do not get worse in the case of unstructured meshes,
despite the lack of a rigorous proof. \newline
%
\subsection{Numerical results} \label{sez:numerical_results}
Before analyzing in detail selected numerical results, we wish to
give technical information on the performed numerical experiments.
We apply the PCG or the PGMRES method, in the symmetric and
non-symmetric case, respectively, with the preconditioning
strategy described in section \ref{sez:fem}, to FE approximations
of the problem (\ref{eq:modello}). Whenever required, the involved
integrals have been approximated by means of the barycentric quadrature 
rule (the approximation by means of the nodal quadrature 
rule gives rise to similar results, indeed both are exact when
applied to linear functions).\par The domains of integration
$\Omega$ are those reported in Figure \ref{fig:esagono_esagoni}
and we assume Dirichlet boundary conditions.
All the reported numerical experiments are performed in Matlab, by
employing the available \texttt{pcg} and \texttt{gmres} library
functions; the iterative solvers start with zero initial guess and
the stopping criterion $||r_k||_2\leq 10^{-7}||r_0||_2$ is
considered. The case of unstructured meshes is also discussed and
compared, together with various types of regularity in the
diffusion coefficient. \par In fact, we consider the case of a
coefficient function satisfying the assumptions
(\ref{eq:ipotesi_coefficienti}). More precisely, the second
columns in Table \ref{tab:IT-ES123esagono_triangoliequilateri_M5}
report the number of iterations required to achieve the
convergence for increasing values of the coefficient matrix size
$n=n(h)$ when considering the FE approximation with structured
uniform meshes as in Figure \ref{fig:esagono_esagoni} and with
template function
$\a(x,y)=\a_1(x,y)=\exp(x+y)$, $\b(x,y)=[x\ y]^T$.
The subsequent meshes are obtained by means of a progressive
refinement procedure, consisting in adding new nodes corresponding
to the middle point of each edge. The numerical experiments
plainly confirm the previous theoretical analysis in section
\ref{sez:clustering}: the convergence behavior does not depend on
the coefficient matrix dimension $n=n(h)$. \par
As anticipated, despite the lack of corresponding theoretical results, we want to test the convergence behavior
also in the case in which the regularity assumption on $\a(x,y)$
in Theorems \ref{teo:clu+se_A_esagoni}
and \ref{teo:cluA} are not satisfied. The analysis is motivated by
favorable known numerical results in  the case of FD
approximations (see, for instance,
\cite{ST-ETNA-2003,ST-SIMAX-2003,BGST-NM-2005}) or FE
approximation with only the diffusion term \cite{ST-NA-2001}. More
precisely, we consider as template the $\mathcal{C}^1$ function
$\a(x,y)=\a_2(x,y)=e^{x+|y-y_0|^{3/2}}$, the $\mathcal{C}^0$
function $\a(x,y)=\a_3(x,y)=e^{x+|y-y_0|}$, with $y_0=\sqrt(3)/4$
or $y_0=1/2$.
\par
The number of required iterations is listed in the remaining
columns in Table \ref{tab:IT-ES123esagono_triangoliequilateri_M5}.
Notice that also in these cases with a $\mathcal{C}^1$ or
$\mathcal{C}^0$ diffusion function, the  iteration count does not
depend on the coefficient matrix dimension $n=n(h)$.  The same
results are reported in Table
\ref{tab:IT-ES123quadrato_triangolirettangoli_M1} with respect to
structured uniform meshes on the domain $\Omega=(0,1)^2$. \par
Furthermore, we want to test our proposal in the case of some
unstructured meshes generated by triangle \cite{Triangle} with a
progressive refinement procedure. The first meshes in the
considered mesh sequences are reported in Figures
\ref{fig:Mesh_M6newnew} and \ref{fig:Mesh_M4}, respectively.\par
Tables \ref{tab:IT-ES123esagono_unstructured_M6newnew} and
\ref{tab:IT-ES123quadrato_unstructured_M4} report the number of
required iterations in the case of the previous template
functions. Negligible differences in the iteration trends are
observed for increasing dimensions $n$.
%
%
All these remarks are in perfect agreement with the outliers
analysis of the matrices $P_n^{-1}(\a)A_n(\a,\b)$, with respect to
a cluster at $1 \in \mathbb{C}^+$ (see some examples in Figure
\ref{fig:outlier_analysis}). \par
%
%
To conclude the section we take into consideration the more
realistic setting in which the meshes are generated by a
specialized automatic procedure. According to this point of view,
we have applied a spectral approximation of the matrix sequence
$\{A_n(a)\}$ in terms of product of low-cost matrix structures,
i.e., $A_n(1)$ and $D_n(a)$, that carry the ``structural'' content
of the variational problem  and the specific ``informative''
content contained in the diffusion coefficient function $a$,
respectively. However, the matrix $A_n(1)$ can be still not easy
to handle. Therefore, we  go beyond in such idea by defining a
preconditioner $\tilde P_n(a)$ in which the matrix $A_n(1)$
related to the unstructured mesh is replaced by a suitable
projection of the Toeplitz matrix $T_N(\tilde{f})$,
$\tilde{f}(s_1,s_2)=\sqrt{3}(6-2\cos(s_1)-2\cos(s_2)-2\cos(s_1+s_2)){/3}$,
$(s_1,s_2) \in \mathcal{D}=(-\pi, \pi]^2$ (see Section
\ref{sez:diffusion_equation}).\par The numerical results are
reported in Table \ref{tab:IT_perturbed_mesh}. As expected, the
number of iterations is a constant with respect to the dimension
when the preconditioner $P_n(a)$ is applied. Nevertheless,  for
$n$ large enough, the same seems to be true also for the
preconditioner $\tilde P_n(a)$. Indeed, for increasing dimensions
$n$, the unstructured partitioning is more and more similar to the
one given by equilateral triangles sketched in Fig.
\ref{fig:esagono_esagoni}.a. \par A theoretical ground supporting
these observation is still missing and would be worth in our
opinion to be studied and developed. \par
%
\begin{table}
\caption{Number of PCG and PGMRES iterations - structured meshes as in Fig. \ref{fig:esagono_esagoni}.a, $\Omega$ hexagonal domain.
}
\label{tab:IT-ES123esagono_triangoliequilateri_M5}
\begin{center} \footnotesize
\begin{tabular}{|l|c|c|c|}
 \hline
 \multicolumn{4}{|c|}{PCG \phantom{$\b^T$}}  \\
 \hline
 $n$     & $\a_1(x,y)$ & $\a_2(x,y)$ & $\a_3(x,y)$ \\ \hline
 37      & 3 & 4 & 5 \\
 169     & 3 & 4 & 5 \\
 721     & 3 & 4 & 4 \\
 2977    & 3 & 4 & 4 \\
 12097   & 3 & 4 & 4 \\
 48769   & 3 & 4 & 4 \\
\hline
\end{tabular}
\begin{tabular}{|l|c|c|c|}
 \hline
 \multicolumn{4}{|c|}{PGMRES, $\b(x,y)=[x \ y]^T$}  \\
 \hline
 $n$    & $\a_1(x,y)$ & $\a_2(x,y)$ & $\a_3(x,y)$ \\ \hline
 37      & 4 & 5 & 5 \\
 169     & 4 & 5 & 5 \\
 721     & 4 & 5 & 5 \\
 2977    & 4 & 5 & 5 \\
 12097   & 4 & 5 & 5 \\
 48769   & 4 & 5 & 5 \\
\hline
\end{tabular}
\end{center}
\end{table}
\begin{table}
\caption{Number of PCG and PGMRES iterations - structured meshes as in Fig. \ref{fig:esagono_esagoni}.b, $\Omega=(0,1)^2$. 
}
\label{tab:IT-ES123quadrato_triangolirettangoli_M1}
\begin{center} \footnotesize
\begin{tabular}{|l|c|c|c|}
 \hline
 \multicolumn{4}{|c|}{PCG \phantom{$\b^T$}}  \\
 \hline
 $n$    & $\a_1(x,y)$ & $\a_2(x,y)$ & $\a_3(x,y)$ \\ \hline
 81     & 3 & 4 & 4 \\
 361    & 3 & 4 & 5 \\
 1521   & 3 & 4 & 5 \\
 6241   & 3 & 4 & 5 \\
 25281  & 3 & 4 & 5 \\
 128881 & 3 & 4 & 5 \\

\hline
\end{tabular}
\begin{tabular}{|l|c|c|c|}
 \hline
 \multicolumn{4}{|c|}{PGMRES, $\b(x,y)=[x \ y]^T$}  \\
 \hline
 $n$    & $\a_1(x,y)$ & $\a_2(x,y)$ & $\a_3(x,y)$ \\ \hline
 81     & 4  & 5 & 5 \\
 3611   & 4  & 5 & 5 \\
 1521   & 4  & 5 & 5 \\
 6241   & 4  & 5 & 5 \\
 25281  & 4  & 5 & 5 \\
 128881 & 4  & 5 & 5 \\
\hline
\end{tabular}
\end{center}
\end{table}
\begin{table}
\caption{Number of PCG and PGMRES iterations - unstructured meshes on the hexagonal domain $\Omega$.}
\label{tab:IT-ES123esagono_unstructured_M6newnew}
\begin{center} \footnotesize
\begin{tabular}{|l|c|c|c|}
 \hline
 \multicolumn{4}{|c|}{PCG \phantom{$\b^T$}}  \\
 \hline
 $n$    & $\a_1(x,y)$ & $\a_2(x,y)$ & $\a_3(x,y)$ \\ \hline
 28     & 5 & 5 & 5 \\
 73     & 4 & 4 & 5 \\
 265    & 4 & 4 & 5 \\
 1175   & 4 & 4 & 5 \\
 4732   & 4 & 4 & 5 \\
 19288  & 4 & 4 & 5 \\
 76110  & 4 & 4 & 4 \\
\hline
\end{tabular}
\begin{tabular}{|l|c|c|c|}
 \hline
 \multicolumn{4}{|c|}{PGMRES, $\b(x,y)=[x \ y]^T$}  \\
 \hline
 $n$    & $\a_1(x,y)$ & $\a_2(x,y)$ & $\a_3(x,y)$ \\ \hline
 28     & 4 & 5 & 5 \\
 73     & 4 & 5 & 5 \\
 265    & 4 & 5 & 5 \\
 1175   & 4 & 5 & 5 \\
 4732   & 4 & 5 & 5 \\
 19288  & 4 & 5 & 5 \\
 76110  & 4 & 5 & 5 \\
\hline
\end{tabular}
\end{center}
\end{table}
\begin{table}
\caption{Number of PCG and PGMRES iterations - unstructured meshes on the domain $\Omega=(0,1)^2$.}
\label{tab:IT-ES123quadrato_unstructured_M4}
\begin{center} \footnotesize
\begin{tabular}{|l|c|c|c|}
 \hline
 \multicolumn{4}{|c|}{PCG \phantom{$\b^T$}}  \\
 \hline
 $n$     & $\a_1(x,y)$ & $\a_2(x,y)$ & $\a_3(x,y)$ \\ \hline
 24      & 4 & 5 & 5  \\
 109     & 4 & 5 & 5  \\
 465     & 4 & 5 & 5  \\
 1921    & 4 & 5 & 5  \\
 7809    & 4 & 5 & 5  \\
 31489   & 4 & 5 & 5  \\
 \hline
\end{tabular}
\begin{tabular}{|l|c|c|c|}
 \hline
 \multicolumn{4}{|c|}{PGMRES, $\b(x,y)=[x \ y]^T$}  \\
 \hline
 $n$     & $\a_1(x,y)$ & $\a_2(x,y)$ & $\a_3(x,y)$ \\ \hline
 24      & 4 & 5 & 5 \\
 109     & 4 & 5 & 5  \\
 465     & 4 & 5 & 5  \\
 1921    & 4 & 5 & 5  \\
 7809    & 4 & 5 & 5  \\
 31489   & 4 & 5 & 5  \\
\hline
\end{tabular}
\end{center}
\end{table}
\begin{table}
\caption{Number of PCG and PGMRES iterations - structured and
unstructured meshes on the hexagonal domain $\Omega$,
 $\a_1(x,y)$, $\b(x,y)=[x \ y]^T$.}
\label{tab:IT_perturbed_mesh}
%
\label{tab:IT-ES123esagono_perturbed}
\begin{center} \footnotesize
\begin{tabular}{|l|c|c|}
 \hline
 \multicolumn{3}{|c|}{PCG}  \\
 \hline
 $n$      & $P$ & $\tilde{P}$  \\ \hline
  37     & 4 & 8  \\
 169     & 4 & 10 \\
 721     & 4 & 11 \\
 2977    & 4 & 11  \\
 12095   & 4 & 13  \\
 48769   & 4 & 16  \\
\hline
\end{tabular}
\quad
\begin{tabular}{|l|c|c|}
 \hline
 \multicolumn{3}{|c|}{PGMRES}  \\
 \hline
 $n$      & $P$ & $\tilde{P}$  \\ \hline
 37      & 4 & 8  \\
 169     & 4 & 9  \\
 721     & 4 & 9  \\
 2977    & 4 & 10 \\
 12095   & 4 & 11  \\
 48769   & 4 & 13  \\
\hline
\end{tabular}
\end{center}
\end{table}
\begin{figure}
\centering \vskip -0.3cm
\epsfig{file=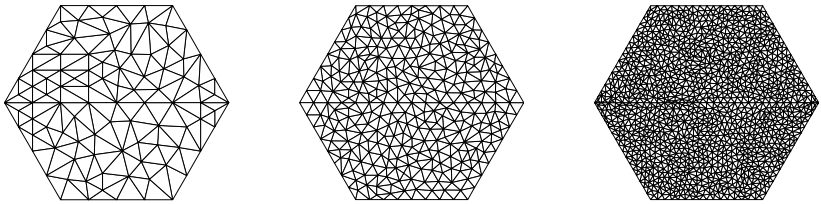,height=12cm} \\
\vskip -4.5cm \caption{Unstructured Meshes on the hexagonal domain
$\Omega$.} \label{fig:Mesh_M6newnew}
\end{figure}
\begin{figure}
\centering \vskip -0.5cm
\epsfig{file=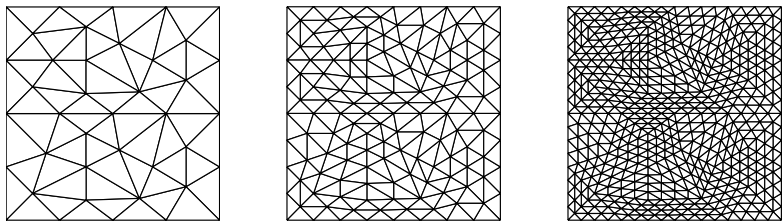,clip=, height=10cm}
\vskip -3.5cm \caption{Unstructured Meshes on the domain
$\Omega=(0,1)^2$.} \label{fig:Mesh_M4}
\end{figure}
\clearpage
\begin{figure}
\centering
\epsfig{file=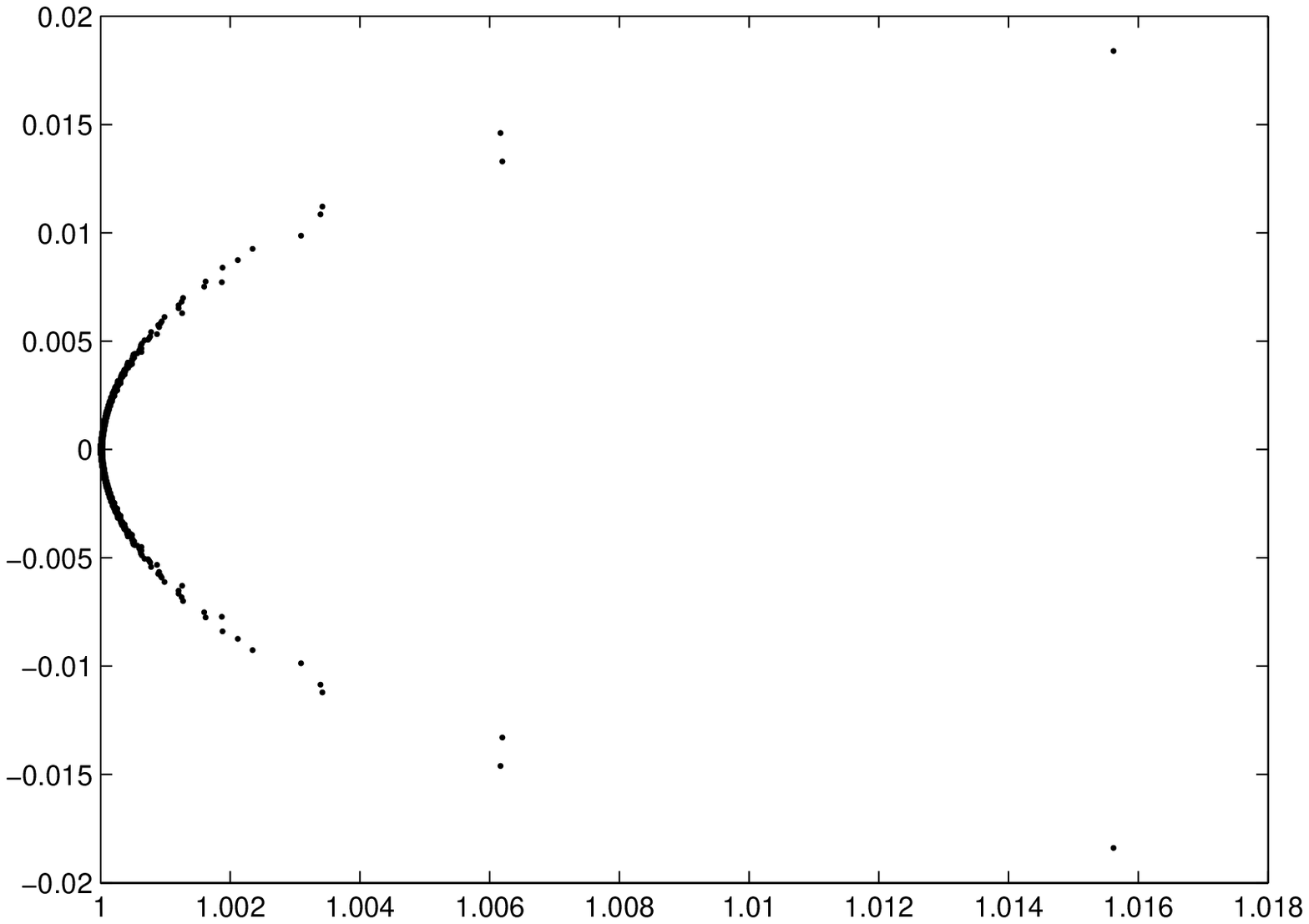,clip=, width=4.5cm} \hskip -0.5cm
\epsfig{file=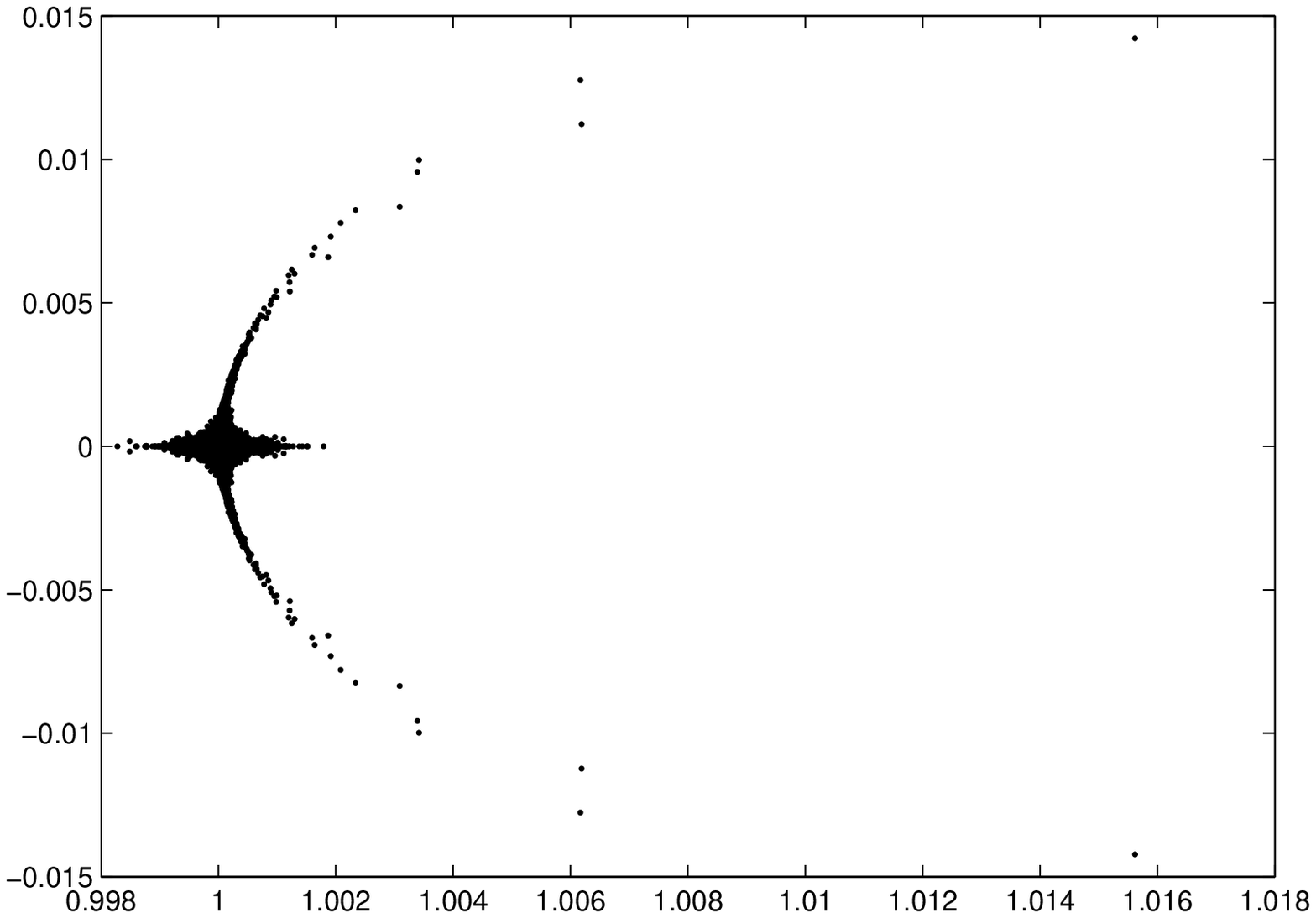,clip=, width=4.5cm}  \hskip -0.5cm
\epsfig{file=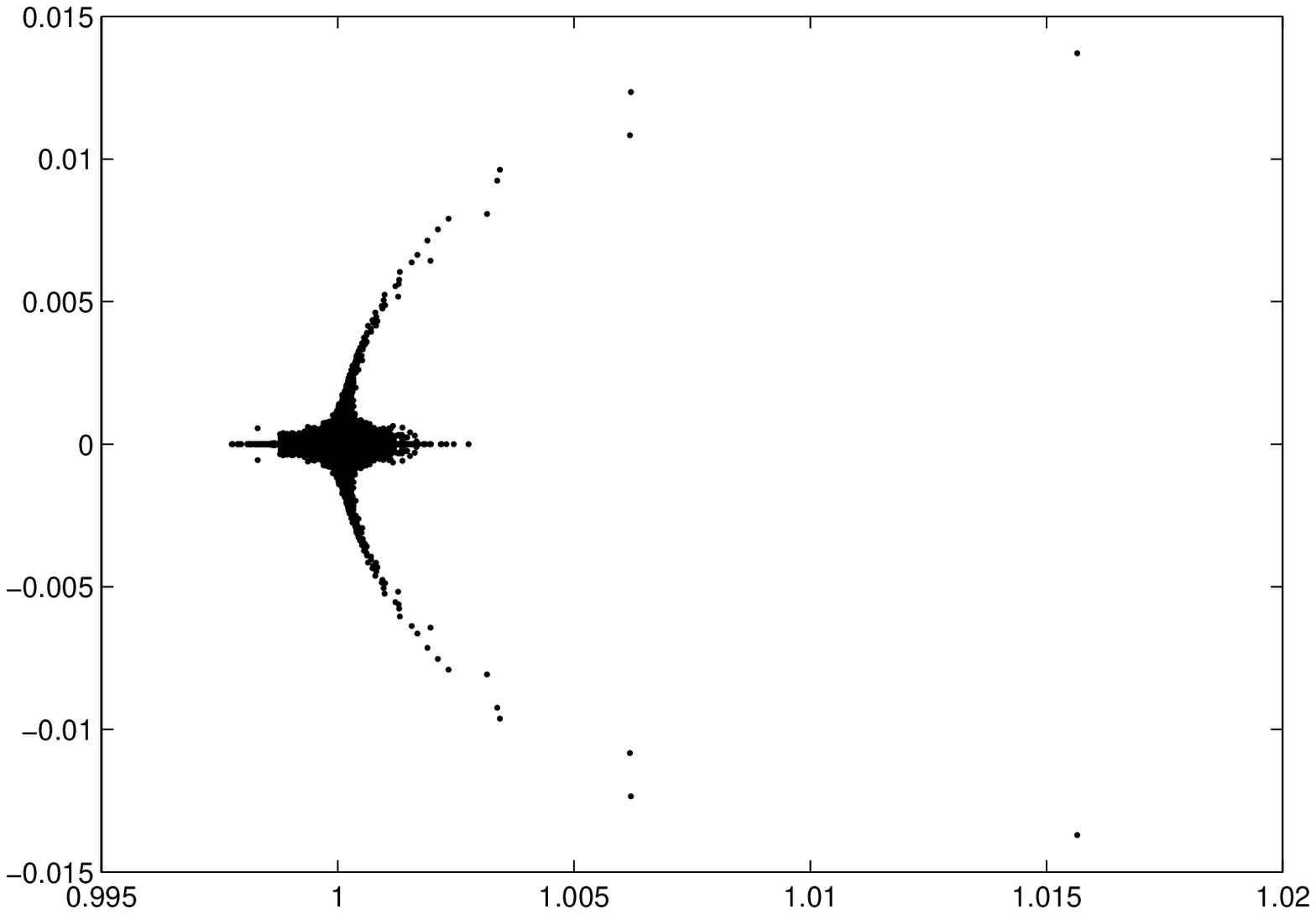,clip=, width=4.5cm}\\
\epsfig{file=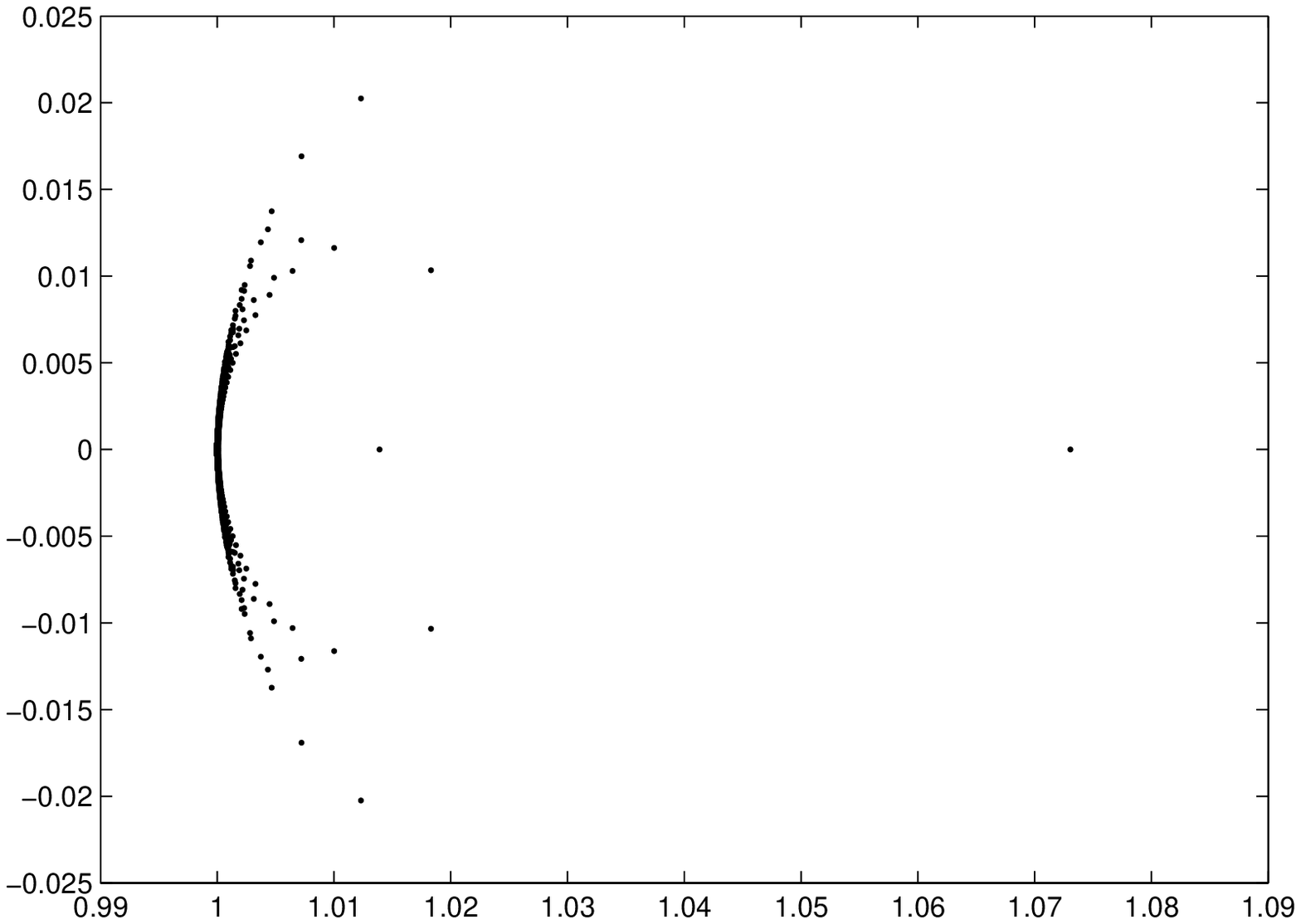,clip=, width=4.5cm} \hskip -0.5cm
\epsfig{file=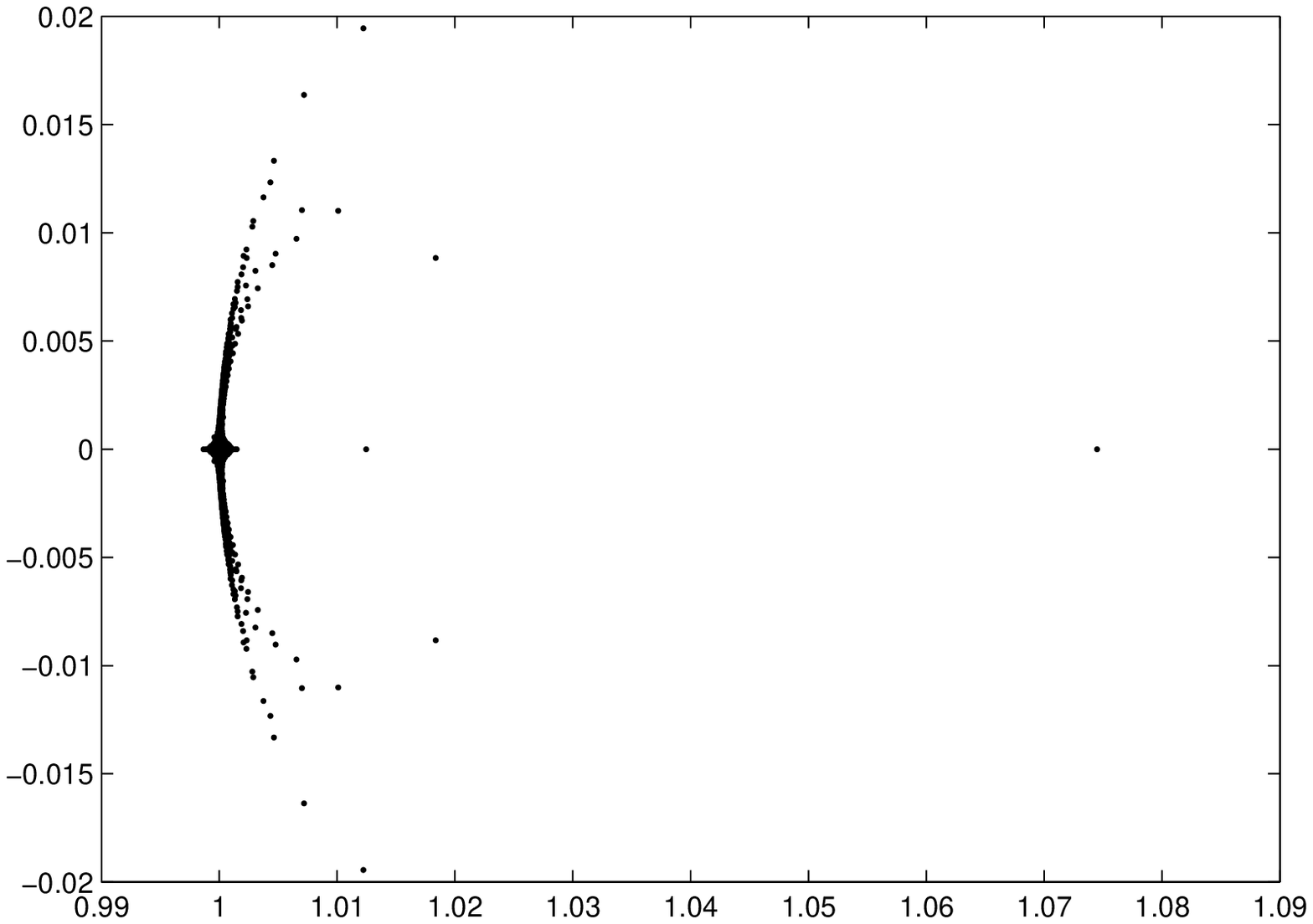,clip=, width=4.5cm}\hskip -0.5cm
\epsfig{file=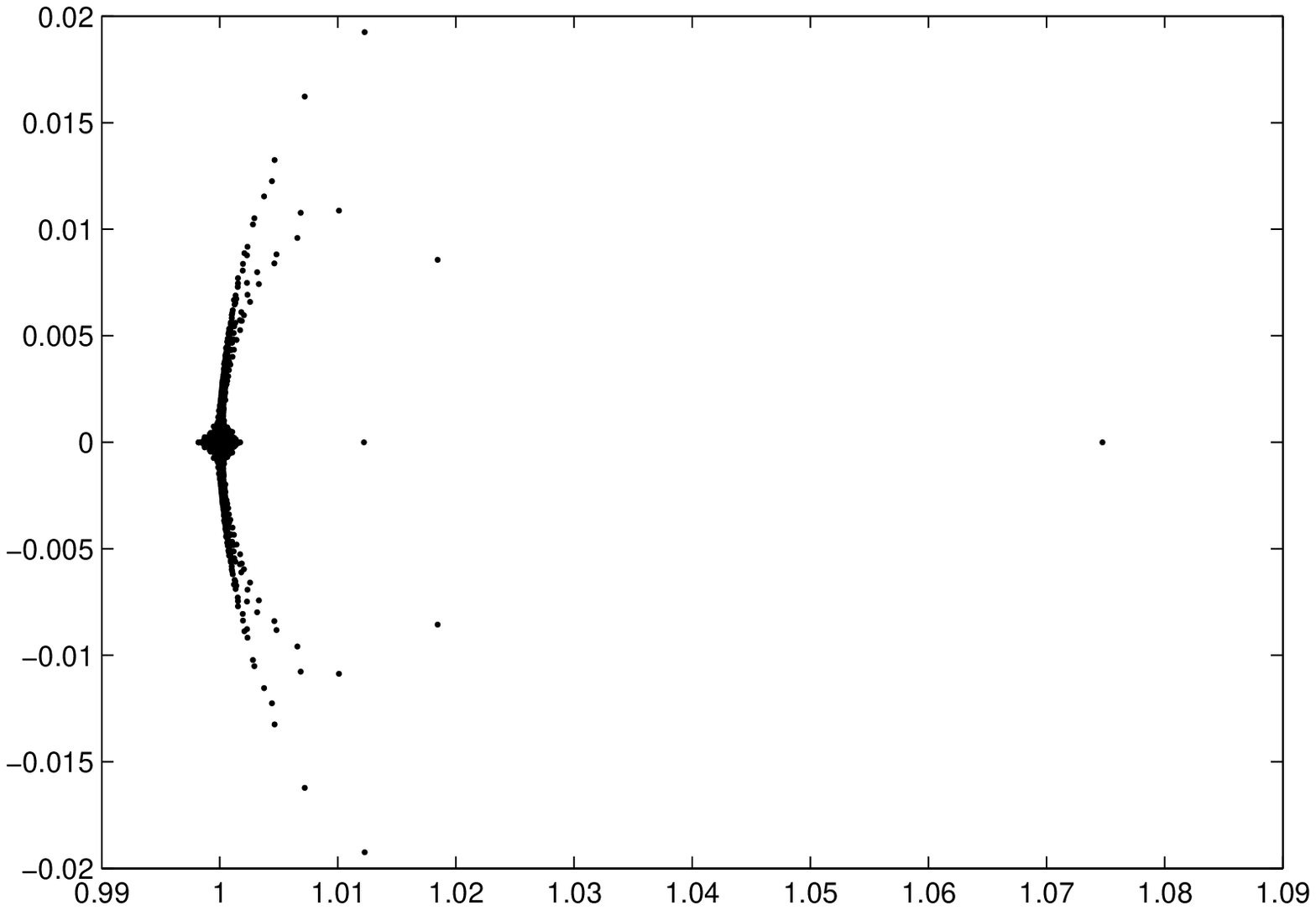,clip=, width=4.5cm} \\
\epsfig{file=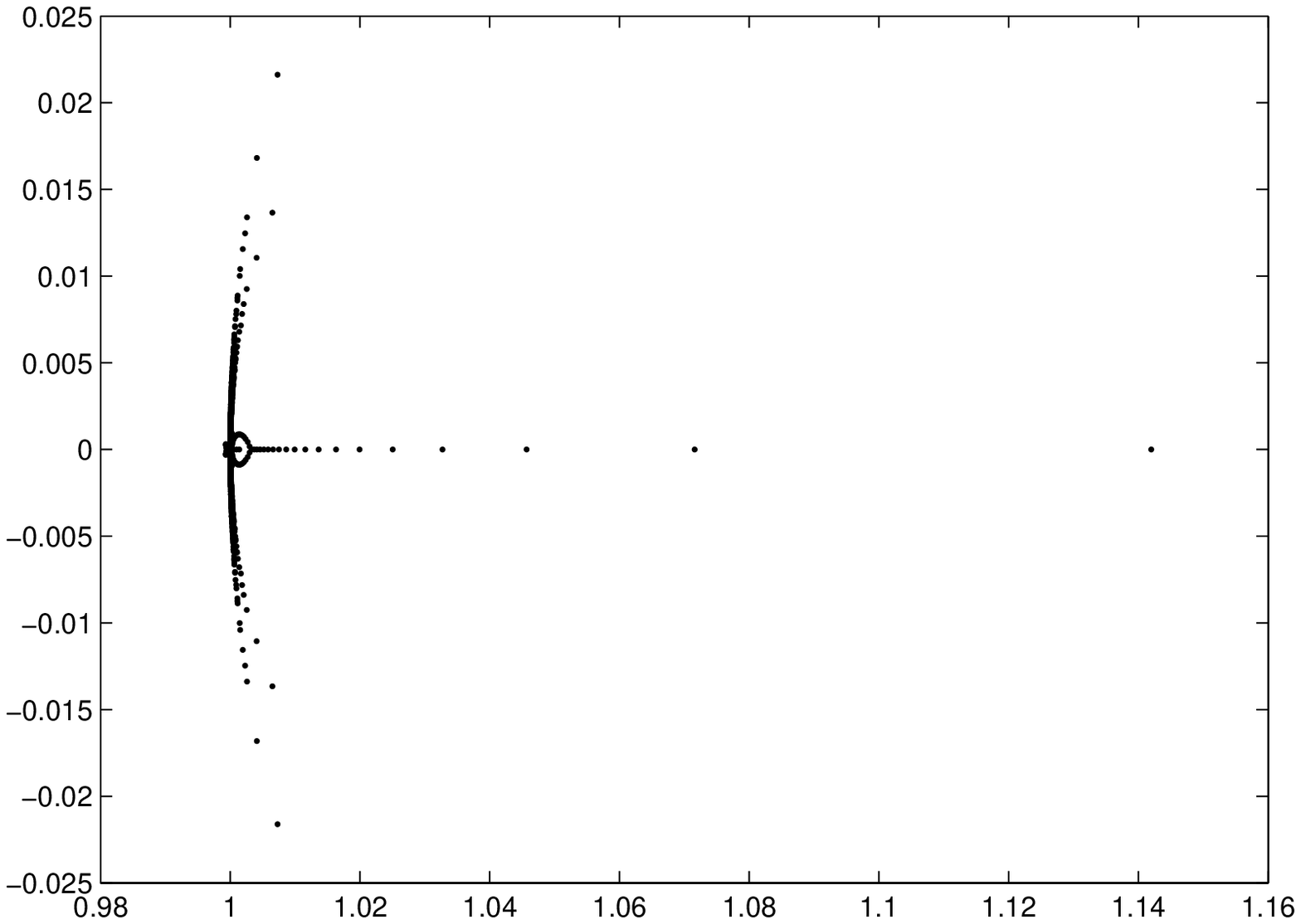,clip=, width=4.5cm} \hskip -0.5cm
\epsfig{file=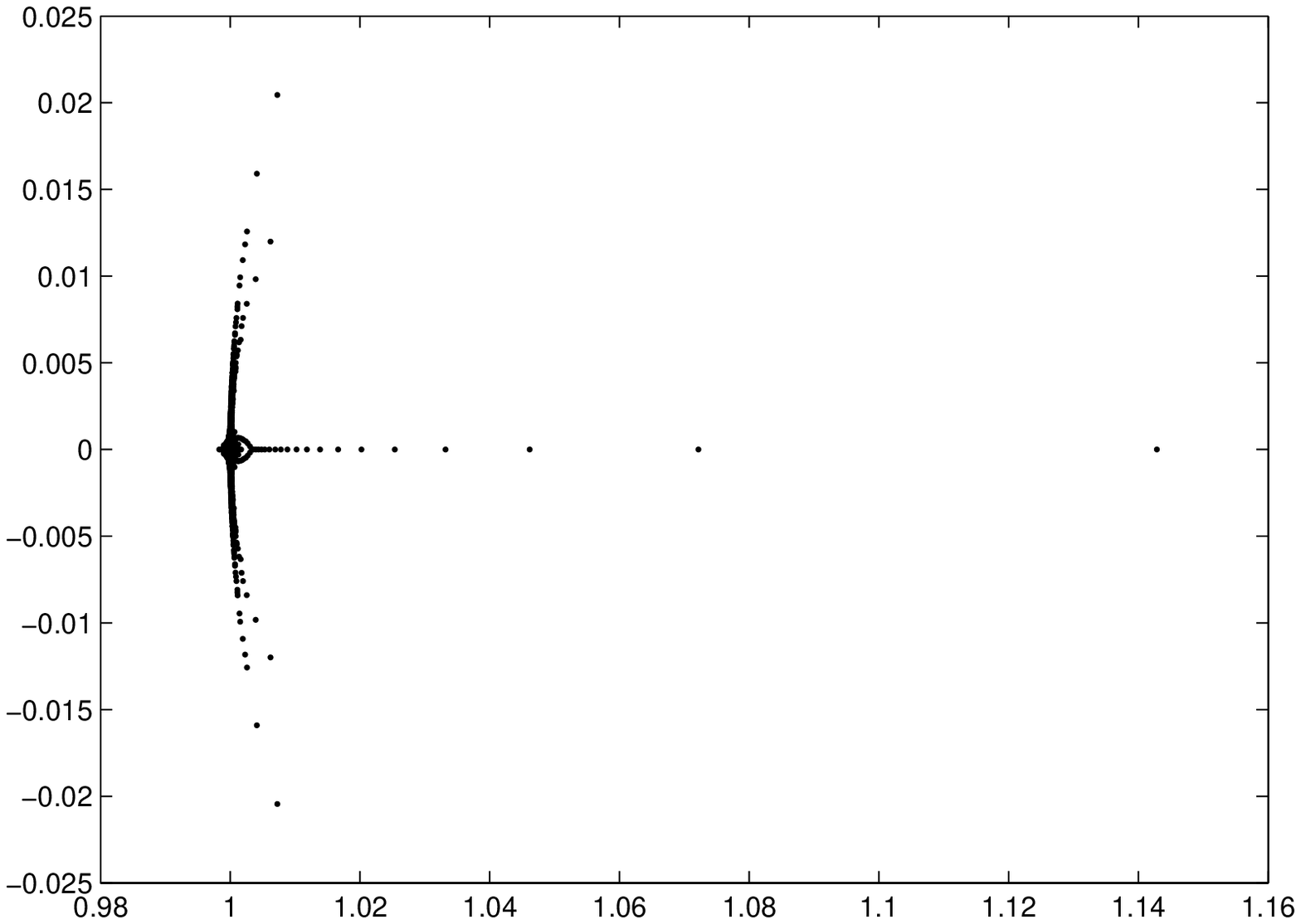,clip=, width=4.5cm}\hskip -0.5cm
\epsfig{file=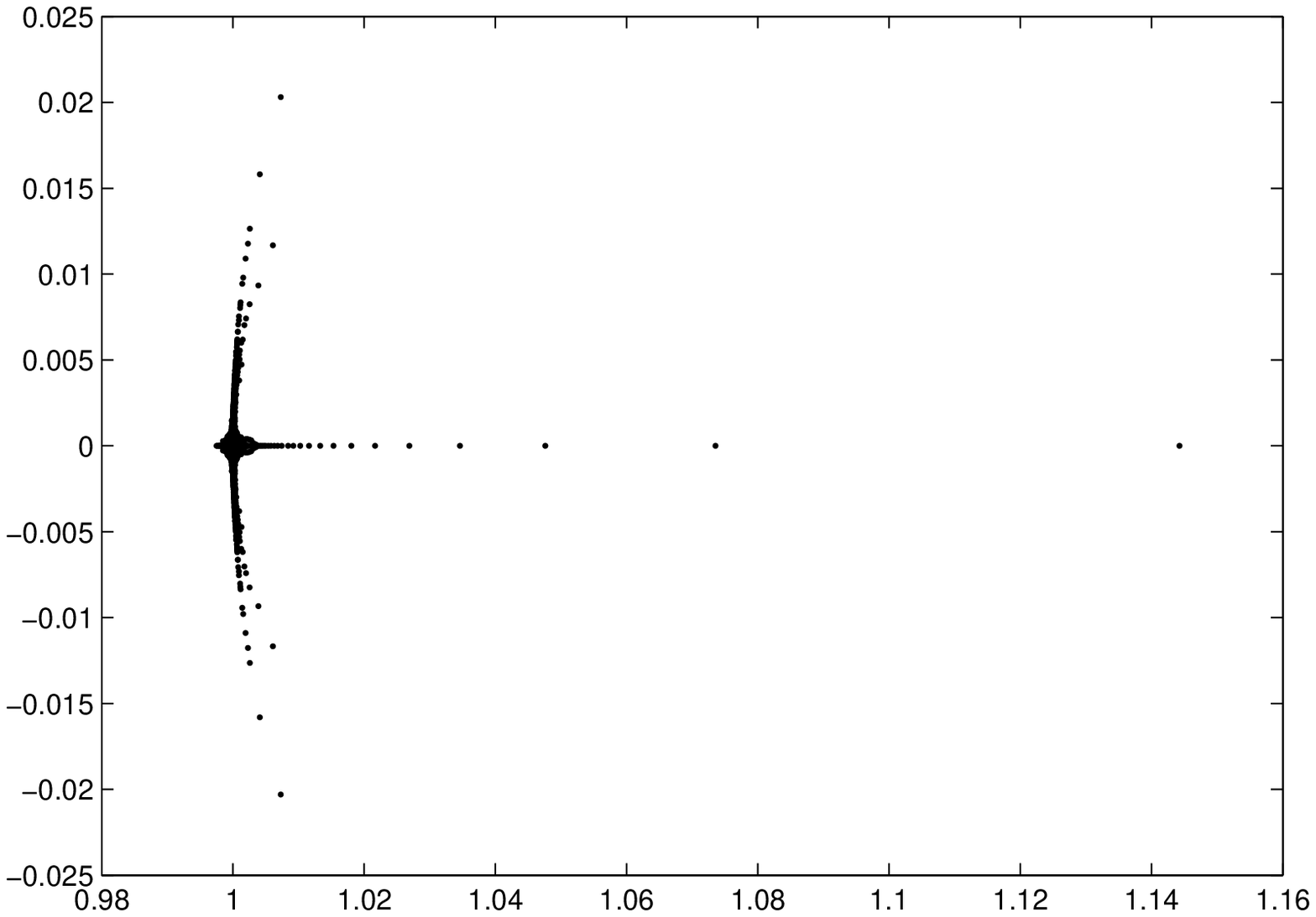,clip=, width=4.5cm}
\caption{Outlier analysis - structured, perturbed and unstructured meshes on the hexagonal domain $\Omega$,
$\a_1(x,y)$, $\a_2(x,y)$, $\a_3(x,y)$, and $\b(x,y)=[x \ y]^T$.}
\label{fig:outlier_analysis}
\end{figure}
\section{Perspectives and future works} \label{sez:conclusions}
%
As emphasized in the introduction it is clear that the problem in
Figure \ref{fig:esagono_esagoni}.a is just an academic example,
due to the perfect structure made by equi-lateral triangles.
However, it is a fact that a professional mesh generator will
produce a partitioning, which  ``asymptotically'', that is a for a
mesh fine enough, tends to the one in Figure
\ref{fig:esagono_esagoni}.a. \par The latter fact has a practical
important counterpart, since the academic preconditioner $\tilde
P_n(a)$ is optimal for the real case with nonconstant coefficients
and with the partitioning  in B). A theoretical ground supporting
these observations is still missing and would be worth in our
opinion to be studied and developed in three directions: a) giving
a formal notion of convergence of a partitioning to a structured
one, b) proving spectral and convergence results in the case of an
asymptotically structured partitioning, c) extending the spectral
analysis in the case of weak regularity assumptions.\par Other
possible developments  include the case of higher order finite
element spaces: it would be intriguing to find an expression of
the underlying Toeplitz symbol as a function of the different
parameters in the considered finite element space (degrees of
freedom, polynomial degree, geometry of the mesh), and this could
be done uniformly in $d$ dimension, i.e. for equation
(\ref{eq:modello}) with $\Omega \subseteq  \mathbf{R}^d$, $d\ge
2$.
%

%
\end{document}